\documentclass[a4paper,11pt]{article}
\usepackage{amsmath,amsthm}
\usepackage{amssymb,mathrsfs}
\usepackage{bm,mathtools}
\usepackage{tikz}
\usetikzlibrary{calc}

\usepackage{xcolor}
\usepackage{authblk}

\usepackage{geometry}
\usepackage[colorlinks=true,
linkcolor=blue,citecolor=cyan,
urlcolor=blue]{hyperref}

\makeatletter
\def\@seccntDot{.}
\def\@seccntformat#1{\csname the#1\endcsname\@seccntDot\hskip 0.5em}
\renewcommand\section{\@startsection{section}{1}{\z@}%
{18\p@ \@plus 6\p@ \@minus 3\p@}%
{9\p@ \@plus 6\p@ \@minus 3\p@}%
{\large\bfseries\boldmath}}
\renewcommand\subsection{\@startsection{subsection}{2}{\z@}%
{12\p@ \@plus 6\p@ \@minus 3\p@}%
{3\p@ \@plus 6\p@ \@minus 3\p@}%
{\bfseries\boldmath}}
\renewcommand\subsubsection{\@startsection{subsubsection}{3}{\z@}%
{12\p@ \@plus 6\p@ \@minus 3\p@}%
{\p@}%
{\bfseries\boldmath}}
\makeatother

\usepackage{microtype}

\theoremstyle{plain}
\newtheorem{theorem}{Theorem}[section]
\newtheorem{lemma}{Lemma}[section]
\newtheorem{corollary}{Corollary}[section]

\usepackage[numbers,sort&compress]{natbib}
\theoremstyle{definition}

\numberwithin{equation}{section}
\allowdisplaybreaks
\title{Characterizations of the graphs with dominating parameters}
  \author{Yuhan Ma\footnote{ Email: kid4068@163.com.}
}
\affil{{\footnotesize School of Mathematical Sciences, East China Normal University, Shanghai, 200241, China}}

\date{}

\begin{document}
\maketitle
\begin{abstract}
	A subset $S$ of vertices of $G$ is a \textit{dominating set} of $G$ if every vertex in $V(G)-S$ has a neighbor in $S$. The \textit{domination number} \(\gamma(G)\) is the minimum cardinality of a dominating set of $G$. A dominating set $S$ is a \textit{total dominating set} if $N(S)$=$V$ where $N(S)$ is the neighbor of $S$. The \textit{total domination number} \(\gamma_t(G)\) equals the minimum cardinality of a total dominating set of $G$. A set $D$ is an \textit{isolate set} if the induced subgragh $G[D]$ has at least one isolated vertex. The \textit{isolate number} \(i_0(G)\) is the minimum cardinality of a maximal isolate set. In this paper we study these parameters and answer open problems proposed by Hamid et al. in 2016.   

\par\vspace{2mm}

\noindent{\bfseries 2020 Mathematics Subject Classification:} 
05C38, 05C35, 05C40
\par\vspace{2mm}
\noindent{\bfseries Keywords:} Domination; Isolate; Total domination
\end{abstract}

\section{Introduction}
By a graph $G=(V,E)$, we mean a finite, non-trivial, undirected graph with neither loops nor multiple edges. For graph theoretic terminology we refer to the book by Haynes, Hedetniemi and Slater \cite{TSP}.

Let $G$ be a graph with vertex set V of order $|V|= n$ and size $|E|= m$, and let \(\upsilon\) be a vertex in $V$. The \textit{open neighborhood} of \(\upsilon\) is \textit{\(N_G\)}(\(\upsilon\))= \{\(u\in V\mid uv\in E(G)\)\} and the \textit{closed neighborhood} of \(\upsilon\) is \textit{\(N_G\)}[\(\upsilon\)]= \{\(\upsilon\)\} \(\cup\) \textit{\(N_G\)}(\(\upsilon\)). The degree of \(\upsilon\) is $deg_G\)(\(\upsilon\))= \( |\textit{\(N_G\)}(\upsilon)|$. If the graph G is clear from the context, we simply write \textit{N}(\(\upsilon\)) and deg(\(\upsilon\)). For a set $S$\(\subseteq\)$V$, the \textit{open neighborhood} of $S$ is the set $N(S)$= \(\cup_{\upsilon\in S}\)\textit{N}(\(\upsilon\)), and the \textit{closed neighborhood} of $S$ is the set $N[S]$= $N(S)$ \(\cup\) $S$. A vertex of degree one is called a \textit{leaf} and its unique neighbor is called a \textit{support vertex}. For a subset $S$ of vertices of $G$, we denote by $G[S]$ the subgraph of $G$ \textit{induced} by $S$. The \textit{diameter} of a graph $G$, denoted by \textit{diam}($G$), is the maximum distance between pairs of vertices of $G$.

A vertex $v$ is called a \textit{dominating vertex} if $v$ is adjacent to every other vertex in $G$. A set \( D \) of vertices of a graph \( G \) is said to be a \textit{dominating set} if every vertex in \( V\setminus D \) is adjacent to a vertex in \( D \). A dominating set \( D \) is said to be a \textit{minimal dominating set} if no proper subset of \( D \) is a dominating set. The minimum cardinality of a minimal dominating set of a graph \( G \) is called the \textit{domination number} of \( G \) and is denoted by \( \gamma(G) \). The \textit{upper domination number} \( \Gamma(G) \) is the maximum cardinality of a minimal dominating set of \( G \). The minimum cardinality of an independent dominating set is called the \textit{independent domination number}, denoted by \( i(G) \) and the \textit{independence number} \( \beta_0(G) \) is the maximum cardinality of an independent set of \( G \). A set \( S \) is a \textit{total dominating set} if \( N(S) = V \). The \textit{total domination number} \( \gamma_t(G) \) equals the minimum cardinality of a total dominating set of \( G \). A set \( D \subseteq V(G) \) which is a dominating set of both \( G \) and \( \overline{G} \) is called a \textit{global dominating set}. The minimum cardinality of a global dominating is called the \textit{global domination number} and is denoted by \( \gamma_g(G) \). A set \( S \) of vertices is \textit{irredundant} if every vertex \( v \in S \) has at least one private neighbor. The minimum and maximum cardinality of a maximal irredundant set are respectively called the \textit{irredundance number} \( ir(G) \) and the \textit{upper irredundance number} \( IR(G) \).

A vertex $v$ is called an \textit{isolated vertex} if $deg(v)=0$. A set \( S \) of vertices of a graph \( G \) such that $G[S]$ has an isolated vertex is called an \textit{isolate set} of \( G \). The minimum and maximum cardinality of a maximal isolate set are called the \textit{isolate number} \( i_0(G) \) and the \textit{upper isolate number} \( I_0(G) \). An isolate set that is also a dominating set (an irredundant set) is an \textit{isolate dominating set} (an \textit{isolate irredundant set}). The \textit{isolate domination number} \( \gamma_0(G) \) and the \textit{upper isolate domination number} \( \Gamma_0(G) \) are respectively the minimum and maximum cardinality of a minimal isolate dominating set while the \textit{isolate irredundance number} \( ir_0(G) \) and the \textit{upper isolate irredundance number} \( IR_0(G) \) are the minimum and maximum cardinality of a maximal isolate irredundant set of \( G \). An isolate set \( S \) of \( G \) with \( |S| = i_0(G) \) is called an \( i_0 \)-set of \( G \). Similarly, $\gamma_t$-set, \( \gamma_0 \)-set, \( \Gamma_0 \)-set, \( ir_0 \)-set are defined. If a vertex is adjacent to all other vertices, then we call it a \textit{dominating vertex}. 

Up to now, many scholars have made excellent conclusions on the parameters of various dominating sets and their variants:
\begin{theorem}\cite{BE}
    For any graph $G$, $\frac{\gamma(G)}{2}\leq ir(G)\leq \gamma(G)\leq 2ir(G)-1$.
\end{theorem}
\begin{theorem}\cite{196}
    For every graph \( G \),
\[
\text{ir}(G) \leq \gamma(G) \leq i(G) \leq \alpha(G) \leq \Gamma(G) \leq \text{IR}(G).
\]
\end{theorem}
\begin{theorem}\cite{185}
    If \( G \) is a bipartite graph, then
\[
\text{ir}(G) \leq \gamma(G) \leq i(G) \leq \alpha(G) = \Gamma(G) = \text{IR}(G).
\]
\end{theorem}
This paper further studies these concepts by establishing some characterizations of graphs with fixed condition of these parameters.

\section{On graphs $G$ with \(\gamma_t(G)\)\ = \(i_0(G)\)+1}\label{Pre}

Hamid, Balamurugan and Navaneethakrishnan \cite{Note} have proposed the following theorem to study the above concepts:
\begin{theorem}\label{con1}
	For any graph $G$, \(\gamma_t(G)\leq\)\ \(i_0(G)\)+1 and the bound is sharp.
\end{theorem}

In \cite{Note},  the author Hamid demonstrated a specific upper bound for the total dominating number $\gamma_t(G)$ of a graph using star graphs, namely $\gamma_t(G) = i_0(G) + 1$, and on this basis proposed a classification problem concerning such graphs. To thoroughly investigate and address this issue, we need to reference another crucial conclusion:
\begin{theorem}\cite{Chain}\label{lem}
	Let $S$ be an isolate set of a graph $G$. Then, $S$ is a maximal isolate set of $G$ if and only if every vertex in $V\setminus S$ is adjacent to all the isolates of $S$.
\end{theorem}
Now, we have the following theorem.

\begin{theorem}\label{diam}
    If a connected graph $G$ satisfies \(\gamma_t(G)=i_0(G)\)+1, then $diam(G)$ \(\leq2\).
\end{theorem}

\begin{proof}
	To the contrary, suppose $G$ is a graph which satisfies \(\gamma_t(G)=i_0(G)\)+1 and $diam(G)$ \(\geq3\). Then there exists a diametral path $P=v_1v_2...v_n$ where $n$ \(\geq4\).

Firstly, we claim that every \(i_0(G)\)-set $S$ contains some vertex in $P$. Otherwise, by Theorem~\ref{lem}, every vertex in $P$ is adjacent to every isolated vertex in $G[S]$. This contradicts to the select of diametral path $P$. 

Then, we assume that there exist some vertices of $P$ which are contained in a \(i_0(G)\)-set $S$ while there are still some vertices of $P$ outside of $S$. Since $diam(G)$ \(\geq3\), there must exist \(v_i,v_j\in V(P)\) such that $v_i\in S$ and $v_j\in V\setminus S$ with \(d(v_i,v_j)\geq2\) in $P$. Then by Theorem~\ref{lem} we know that these two vertices are adjacent, which is impossible since $P$ is an induced subpath.

Finally, we can assume that \(V(P)\subseteq S\). By the proof of Theorem~\ref{con1} in \cite{Note}, for any \(v\in V\setminus S\), \(S\cup \{v\}\) is a total dominating set. And since \(\gamma_t(G)=i_0(G)\)+1, \(S\cup \{v\}\) is a minimum total dominating set. Now  we consider any vertex adjacent to \(v_1\). Since $P$ is a diametral path, every vertex in $N$(\(v_1\)) must be adjacent to some vertices in $V(P)$\text{-}\(\{v_1, v_n\}\). Hence $(S-\left\{v_1,v_n\right\})$\(\cup \{v\}\) is a total dominating set with cardinality less than \(i_0(G)+1\), where $v$ is a neighbor of $v_1$. This leads a contradiction to the definition of the total domination number.

Now we can assume that $deg(v_1)=deg(v_n)=1$. But now since $G[S]$ contains an isolated vertex, a similar discussion will lead to a contradiction. 

Above all, any graph $G$ which satisfies  \(\gamma_t(G)=i_0(G)\)+1, its diameter  is at most 2.
\end{proof}

Applying this theorem, we have the following conclusion:
\begin{theorem}\label{dv}
	Every connected graph G which satisfies \(\gamma_t(G)=i_0(G)\)+1 must have a dominating vertex. 
\end{theorem}
\begin{proof}
     Let $G$ be any connected graph which satisfies  \(\gamma_t(G)=i_0(G)\)+1. By Theorem~\ref{diam} we know that $diam(G)$ \(\leq 2\).

If $diam(G)=0$, then $G$ is trivial.

If $diam(G)=1$, we claim that $G$ is complete, otherwise there exist two vertices $u$ and $v$ are nonadjacent to each other. Hence \(d_G(u,v)\geq 2\), which contradicts to the diameter of $G$. Now $G$ is complete and obviously satisfies \(\gamma_t(G)=i_0(G)\)+1. Every vertex in $G$ is a dominating vertex.

If $diam(G)=2$, we consider the cardinality of an \(i_0(G)\)-set $S$. Suppose to the contrary that $G$ contains no dominating vertex i.e. \(i_0(G)\) \(\geq 2\). And then we have the following cases:

\textbf{Case 1}. $S$ contains more than one isolated vertex in $G[S]$.

Without loss of generality, we assume that $u$ and $v$ are two isolated vertices in $G[S]$. By Theorem~\ref{lem}, $u$ and $v$ are adjacent to every vertex in \(V\setminus S\). Select an arbitrary vertex \(w\in V\setminus S\). We can easily see that \((S-\{v\})\cup \{w\}\) is a total dominating set. Hence \(\gamma_t(G)\leq i_0(G)<i_0(G)+1\), which is a controdiction.

\textbf{Case 2}. $S$ contains only one isolated vertex $v$ in $G[S]$.

Since $G$ is connected, we select one component $Q$ \(\subseteq G[S]\) to discuss. If there is a vertex $w$ which is not a cut-vertex of $Q$, then select an arbitrary \(x\in V\setminus S\) so we have a total dominating set \((S-\{w\})\cup \{x\}\) with cardinality equals $S$. This is impossible. Since there is only one isolated vertex in $G[S]$, we can always find the above vertex we want. Note that if $Q$ is \(K_2\), then we can select an arbitrary endpoint and repeat the operation above. Therefore, we can always construct a total dominating set with cardinality equals \(i_0(G)\) which leads to a contradiction.

Above all, we have that every connected graph which satisfies the condition  \(\gamma_t(G)=i_0(G)\)+1 must have a dominating vertex.
\end{proof}
Then consider a graph $G$ that has a dominating vertex, we can easily get that \(\gamma_t(G)\)= 2 and \(i_0(G)\)= 1 and so \(\gamma_t(G)=i_0(G)\)+1. Hence we can give the following characterization:
\begin{corollary}
    A connected graph $G$ satisfies  \(\gamma_t(G)=i_0(G)\)+1 if and only if $G$ contains a dominating vertex.
\end{corollary}

\section{On trees $T$ with \texorpdfstring{$\gamma_t(T)=n-l$}{gamma_t(T)=n-l}}
Trees are perhaps the simplest of all graph families, so it is not surprising that much research has been done involving domination in trees. In \cite{ID}, Hamid and Balamurugan said that the study of parameters in domination for trees would be interesting. Hence in this section, we focus our discussion on trees.

We begin with a simple observation about dominating sets in trees. If a \(\gamma\)-set in a tree $T$ of order \(n\geq 3\) contains a leaf, then we can simply replace this leaf in $S$ with its support vertex to produce another \(\gamma\)-set of $T$. Thus, for a tree $T$ of order \(n\geq 3\) with $l$ leaves, we have the following results:
\begin{itemize}
    \item[(a)] There is a \(\gamma\)-set of $T$ that contains no leaf.
    \item[(b)] \(\gamma(T)\leq n-l\).
\end{itemize}

Naturally, we would like to consider the case for trees $T$ when \(\gamma(T)= n-l\).

A \textit{caterpillar} is a tree with the property that removing all of its leaves forms a path. This path is called the \textit{spine} of this tree. The \textit{code} of the caterpillar having spine \(P_k\) : \(v_1 v_2... v_k\) is the ordered $k$-tuple (\(l_1,l_2,...,l_k\)), where \(l_i\) is the number of leaves in the caterpillar adjacent to \(v_i\) for \(i\in [k]\). Through this concept, we have:
\begin{theorem}
    If $T$ is a tree of order \(n\geq 3\) and $l$ leaves. Then \(\gamma(T)= n-l\) if and only if every vertex in $T$ is either a leaf or a support vertex.
\end{theorem}
\begin{proof}
	Suppose $T$ is a tree such that every vertex in $T$ is either a leaf or a support vertex. In order to dominate all leaves, every support vertex must be choosen into the $\gamma(T)$-set and hence $\gamma(T)\ge n-l$. Since every vertex in $T$ is either a leaf or a support vertex, the vertex set containing all support vertex is a dominating set of $T$ and hence \(\gamma(T)\leq n-l\). Therefore, $\gamma(T)=n-l$.

Conversely, suppose there is a tree $T$ of order \(n\geq 3\) and $l$ leaves with \(\gamma(T)= n-l\). By result (a) above, there is a \(\gamma\)-set of $T$ that contains no leaf. Let $L$ denote the set of all leaves in $T$ and hence $V(T)\setminus L$ is a \(\gamma\)-set of $T$. The induced subtree of $V(T)\setminus L$ has no isolated vertex except when $V(T)\setminus L$ is a single vertex and this case is done. Denote $Q=V(T)\setminus L$, since $V(T)\setminus L$ is a \(\gamma\)-set of $T$, $Q$ -\(\{v\}\) is is not a dominating set of $T$ for an arbitrary \(v\in Q\). Hence $v$ must be adjacent to some vertices that $Q$ -\(\{v\}\) do not dominate. Hence these vertices must be leaves and their support vertex is $v$. By the randomicity of $v$, every vertex in $V(T)-L$ is adjacent to some vertices in $L$. Hence every vertex in $T$ is either a leaf or a support vertex.
\end{proof}

In addition, the discussion about total domination number is also valuable and in \cite{Dela} DeLaVi$\widetilde{n}$a \textit{et al} give us a lower bound of total domination number in graphs:
\begin{theorem}\label{Tdiam}
    If $G$ is a nontrivial connected graph, then \(\gamma_t(G)\geq \frac{1}{2}(diam(G)+1)\).
\end{theorem}

The following Lemma can be found in \cite{Dela} and we will use this lemma to characterize the graphs with \(\gamma_t(G)= \frac{1}{2}(diam(G)+1)\).
\begin{lemma}\label{lem'}
    If $D$ is a \(\gamma_t\)-set in a nontrivial connected graph $G$, then there exists a spanning tree $T$ of $G$ such that the following properties hold:
    \item[(1)] $D$ is a \(\gamma_t\)-set in $T$, and
    \item[(2)] the number of components in $T[D]$ equals the number of components in $G[D]$.
\end{lemma}
\begin{theorem}
    If a nontrivial connected graph $G$ satisfies \(\gamma_t(G)= \frac{1}{2}(diam(G)+1)\), then $G$ has a spanning subtree $T$ which is a caterpillar satisfies $diam(T)= diam(G)$.
\end{theorem}

\begin{proof}
	Let $D$ be a \(\gamma_t\)-set of $G$, and let $k$ be the number of components in $G[D]$. By Lemma~\ref{lem'}, there exists a spanning tree $T$ of $G$ such that $D$ is a \(\gamma_t\)-set of $T$ and the number of components in $T[D]$ equals $k$. Since $D$ is a total dominating set of $T$, every component in $T[D]$ contains at least two vertices. Thus, \(\gamma_t(G)\) \(\geq\) $2k$. Let $P$ be a longest path in $T$. Let \(p_1\) denote the number of edges of $P$ in $T[D]$ and \(p_2\) denote the number of edges of $P$ between $D$ and $V\setminus D$ and let \(p_3\) denote the number of edges of $P$ in $T[V\setminus D]$.

 Since $T[D]$ is a forest with $k$ components, $T[D]$ contains \(\gamma_t(T)\)-$k$ edges and hence \(p_1\leq\) \(\gamma_t(T)\)-$k$. In traversing a longest path in $T$, we can enter and leave each component of $T[D]$ at most once. Hence \(p_2\leq\) $2k$. For the value of \(p_3\), let \(m_1\) denote the number of edges in $T[D]$ and \(m_2\) denote the number of edges between $D$ and $V\setminus D$ and let \(m_3\) denote the number of edges in $T[V\setminus D]$. By the discussion above, \(m_1\)= \(\gamma_t(T)\)-$k$. Since $D$ is a total dominating set of $T$, every vertex in $V\setminus D$ is adjacent at least one vertex in $D$ and hence \(m_2\geq\) $n$-\(\gamma_t(T)\). Then by the equation $n-1$=\(m_1+m_2+m_3\) we have \(p_3\leq m_3\leq\) $k-1$. Thus $diam(T)$=\(p_1+p_2+p_3\)\(\leq\) \(\gamma_t(T)\)+$2k-1$\(\leq\) 2\(\gamma_t(T)\)-1. Therefore we have $diam(G)$\(\leq\) $diam(T)$\(\leq\) 2\(\gamma_t(T)\)-1.

 Now we have \(\gamma_t(G)= \frac{1}{2}(diam(G)+1)\), hence $diam(G)=diam(T)$=2\(\gamma_t(T)\)-1. This implies $diam(T)$=\(p_1+p_2+p_3\) =\(\leq\) 2\(\gamma_t(T)\)-1. Therefore $diam(T)$=\(p_1+p_2+p_3\)=2\(\gamma_t(T)\)-1 and we have the equality on the above \(p_i\), $i=1,2,3$. For \(p_1\)=\(\gamma_t(T)\)-$k$, this implies that every component in $T[D]$ is a path and the $k$ paths are all contained in $P$. For \(p_3\)=$k-1$, this implies \(m_3\)=$k-1$ and hence \(m_2\)=$n$-\(\gamma_t(T)\). Since every vertex in $V\setminus D$ is adjacent at least one vertex in $D$, every vertex in $V\setminus D$ is adjacent to exactly one vertex in $D$. Since the $k$ paths in $T[D]$ are all contained in $P$, there are $2k$ vertices of $V\setminus D$ such that we can partite them into $k$ pairs with each pair corresponds to one of the $k$ paths in $T[D]$ and link the two ends of the path, respectively. Now, we have $k$ new paths. Since $T$ is a tree, the $k$ new paths are linked by the edges between different pairs and this needs $k-1$ edges. Hence every remaining vertices in $V\setminus D$(possibly none) is adjacent to exactly one vertex in $D$ and this implies that $T$ is a caterpillar.
\end{proof}

\section{Conclusion}
This paper successfully provides characterizations of the parameters for domination. These discoveries not only enrich the theoretical framework of graph theory but also provide new perspectives to consider the graph containing a dominating vertex and other parameters of the caterpillar.

\section*{Declaration}

\noindent$\textbf{Conflict~of~interest}$
The author declare that he has no known competing financial interests or personal relationships that could have appeared to influence the work reported in this paper.

\noindent$\textbf{Data~availability}$
Data sharing not applicable to this paper as no datasets were generated or analysed during the current study.

\end{document}